\documentclass[11pt]{amsart}
\usepackage{amsfonts,epsfig}
\usepackage{amssymb}

\theoremstyle{plain}
\newtheorem{theorem}{Theorem}
\newtheorem{corollary}{Corollary}

\theoremstyle{definition}

\theoremstyle{remark}

\numberwithin{equation}{section}

\setbox0=\hbox{$+$}
\newdimen\plusheight
\plusheight=\ht0
\def\+{\;\lower\plusheight\hbox{$+$}\;}

\setbox0=\hbox{$-$}
\newdimen\minusheight
\minusheight=\ht0
\def\-{\;\lower\minusheight\hbox{$-$}\;}

\setbox0=\hbox{$\cdots$}
\newdimen\cdotsheight
\cdotsheight=\plusheight
\def\cds{\lower\cdotsheight\hbox{$\cdots$}}

\begin{document}
\bibliographystyle{plain}

\title[Khovanskii's Matrix Methods for extracting Roots ]
 {Some Observations on Khovanskii's Matrix Methods for extracting  Roots
of Polynomials}
\author{J. Mc Laughlin}
\address{Mathematics Department\\
 Trinity College\\
300 Summit Street, Hartford, CT 06106-3100}
\email{james.mclaughlin@trincoll.edu}
\author{ B. Sury}
\address{Stat-Math Unit\\
Indian Statistical Institute\\
8th Mile Mysore Road\\
Bangalore 560 059\\
India.\\
}
\email{sury@isibang.ac.in}
\keywords{rational approximation to  roots of polynomials}

\begin{abstract}
In this article we apply a formula for   the $n$-th power of a
$3\times 3$ matrix
(found previously by the authors) to investigate a procedure
of Khovanskii's for finding the cube root of a positive integer.

We show, for each positive integer $\alpha$, how to construct certain families
of integer sequences such that a certain rational expression,
involving the  ratio of successive terms in each family, tends
to $\alpha^{1/3}$. We also show how to choose the optimal value of a free
parameter to get maximum speed of convergence.

We  apply a similar method, also due to Khovanskii, to a more general
class of cubic equations, and, for each such cubic, obtain
a sequence of rationals that converge to the real root of the cubic.

We prove that Khovanskii's  method for finding the $m$-th
($m \geq 4$) root of a positive integer  works,
provided a free parameter is chosen to satisfy a very simple condition.

Finally, we briefly consider another procedure of Khovanskii's,
which also involves $m \times m$ matrices, for approximating
the root of an arbitrary polynomial of degree $m$.
\end{abstract}

\maketitle

\section{Introduction}
In \cite{K63} Khovanskii described a method which uses
 powers of $3\times3$ matrices to approximate cube roots of integers.
More precisely, let $\alpha$ be a positive integer whose cube root
is desired and let $a$ be an arbitrary integer. Define the matrix
$A$ by
\begin{equation}\label{mateq}
A=
\left (
\begin{array}{ccc}
a & \alpha &\alpha\\
1&a&\alpha\\
1&1&a
\end{array}
\right).
\end{equation}
and let $A_{n,i,j}$ denote the $(i,j)$-th entry of $A^{n}$. Suppose
\begin{align}\label{Aneq}
&\lim_{n \to \infty} \frac{A_{n,1,1}}{A_{n,3,1}}=x, &
&\lim_{n \to \infty} \frac{A_{n,2,1}}{A_{n,3,1}}=y, &
\end{align}
where $x$ and $y$ are finite and $x+y+1 \not = 0$. Then
$x=\sqrt[3]{\alpha^{2}}$ and $y=\sqrt[3]{\alpha}$.

Khovanskii did not give conditions which insure the convergence
of the sequences above. Also, he
did not investigate the speed of convergence or the question of the optimal
choice of the integer $a$ to ensure the most rapid convergence.
Further, there is the difficulty that is necessary to compute the powers
of the matrix $A$.

In this present paper we show that the sequences
$\{A_{n,1,1}/A_{n,3,1}\}_{n=1}^{\infty}$,
$\{A_{n,2,1}/A_{n,3,1}\}_{n=1}^{\infty}$
converge for all  integers $a$ greater than a certain explicit lower bound.
We also
determine, for a given $\alpha$, the choice of $a$ which insures the most
rapid convergence. We also give precise estimates for
$|A_{n,2,1}/A_{n,3,1}-\alpha^{1/3}|$, for this optimal choice
of $a$. Finally, we employ a closed formula
for the $n$-th power of a $3\times 3$ matrix from our paper \cite{McLS04},
which actually makes it unnecessary to perform the matrix multiplications.
We have the following theorems.

\textbf{Theorem \ref{t3}.}
\emph{Let $\alpha>1$  be an
integer and $a$ be any integer such that }
\begin{equation*}
a>-\frac{{\alpha }^{2/3}}
    {1 + {\alpha }^{1/3}}.
\end{equation*}
\emph{Set}
\begin{multline*}
a_n=\\
\sum_{
i,j
}(-1)^i\binom{i+j}{ j}\binom{n-i-2j }{i+j}(3a)^{n-2i-3j}(3a^{2}-3\alpha)^i
(a^3+\alpha -3a\alpha + \alpha^{2})^j.
\end{multline*}
\emph{Then }
\begin{equation*}
\lim_{n \to \infty} 1+
\frac{\alpha -1}
{\displaystyle{
 \frac{a_{n}}{a_{n-1}}-a+1 } }=
\alpha^{1/3}.
\end{equation*}
Note that the limit is independent of the choice of the parameter $a$.

\textbf{Theorem \ref{tdiff}.}
\emph{Let $\alpha$ and $a$ be as described in Theorem \ref{t3}. Let
the matrix $A$ be as described at \eqref{mateq}. Then  the choice of $a$ which
gives the most rapid convergence is one of the two integers closest to}
\begin{equation*}
\bar{a}=  \frac{{\alpha }^{1/3} + \alpha }
  {1 + {\alpha }^{1/3}}.
\end{equation*}
\emph{For this choice of $a$ and $n \geq 3$,}
\begin{equation*}
\frac{A_{n,\,2,\,1}}
{A_{n,\,3,\,1}}-\,\alpha^{1/3}
=\left(
\left(   \omega  -1\right) \,\omega \,
    \left( {\left(\frac{ -\omega}{2}  \right) }^n - {\left( \frac{-{\omega }^2 }{2}\right) }^n \right)
+\frac{n\delta_{3}}{2^{n}\alpha^{1/3}}
+\frac{\delta_{4}}
{2^{2n}}
\right)\alpha^{1/3},
\end{equation*}
\emph{where $\omega = \exp (2 \pi \imath/3)$, $|\delta_{3}|\leq 8$
and $|\delta_{4}|\leq 48$.}

We also investigate two other procedures due to Khovanskii. One
is a method for finding a root of $x^{3}-p\,x-q$ and the other
is a method for finding $\alpha^{1/m}$, where $\alpha$
and $m$ are arbitrary positive
integers.
Again, Khovanskii's methods involve sequences of powers of matrices
and rely on the ratios of certain  matrix entries converging,
and he did not give any conditions which guarantee convergence.
We give criteria which insure convergence.
In the case of $x^{3}-p\,x-q$, we again prove a result which makes the
actual matrix multiplications unnecessary.  We have the following theorems.

\textbf{Theorem \ref{tpq}.}
\emph{Let $p>0$, $q>0$ be integers such that $27q^{2}-4p^{3}>0$. Define }
\[
a_{n}=\sum_{
2i+3j\leq n
}\binom{i+j}{ j}\binom{n-i-2j }{i+j}3^{n-2i-3j}(3-p)^{i}(q-p+1)^{j}.
\]
\emph{Then}
\begin{equation}
-1+\lim_{n \to \infty}\frac{a_{n}}{a_{n-1}}=
\frac{{\left( 2/3\right) }^{1/3}\,p}
   {{\left( 9\,q +{\sqrt{81\,q^2-12\,p^3 }} \right) }^{1/3}} +
  \frac{{\left( 9\,q + {\sqrt{81\,q^2-12\,p^3 }} \right) }^{1/3}}
   {2^{1/3}\,3^{2/3}},
\end{equation}
\emph{the real root of $f(x)=x^{3}-px-q$.}

Let
\begin{equation}\label{matneqa}
A=
\left (
\begin{matrix}
a         & \alpha           &\alpha &\alpha &\dots &\alpha\\
1         &  a                  &\alpha&\alpha &\dots &\alpha\\
1         &          1         &a       &\alpha &\dots &\alpha\\
\vdots & \vdots           & \vdots & \ddots & \vdots & \vdots\\
1   & 1                  &  1        & \dots   &  a      & \alpha \\
1   & 1                  &  1        & \dots   &  1      & a
\end{matrix}
\right).
\end{equation}
\textbf{Theorem \ref{tm}}
\emph{ Let $A$ be the $m \times m$ matrix defined above at \eqref{matneqa}. Let $A_{n,\,i,\,j}$
denote the $(i,\,j)$ entry of $A^{n}$ and suppose $a>0$. Then }
\begin{equation*}
\lim_{n \to \infty}
\frac{A_{n,\,i,\,j}}
{A_{n,\,u,\,v}}
= \alpha^{(j+u-i-v)/m}.
\end{equation*}

Some of the work in this paper relies
heavily on results proved in our paper \cite{McLS04}:
\begin{theorem}\label{t2}
Suppose $A \in M_k(K)$ and let
\[
T^k-s_1T^{k-1}+s_2T^{k-2}+ \cdots +(-1)^ks_k\,I
\]
denote its characteristic polynomial.
Then, for all $n \geq k$, one has
\[
A^n = b_{k-1}A^{k-1}+b_{k-2}A^{k-2}+ \cdots + b_0 \,I,
\]
where
{\allowdisplaybreaks
\begin{align*}
b_{k-1}& = a(n-k+1),\\
b_{k-2} &= a(n-k+2)-s_1a(n-k+1),\\
& \phantom{a}\vdots \\
b_1& = a(n-1)-s_1a(n-2)+ \cdots + (-1)^{k-2}s_{k-2}a(n-k+1),\\
b_0 &= a(n)-s_1a(n-1)+ \cdots + (-1)^{k-1}s_{k-1}a(n-k+1)\\
& = (-1)^{k-1}s_ka(n-k).
\end{align*}
}
and
\[
a(n) = c(i_2,\cdots,i_k,n)
s_1^{n-i_2-2i_3- \cdots -(k-1)i_k}
(-s_2)^{i_2}s_3^{i_3} \cdots ((-1)^{k-1}s_k)^{i_k},
\]
with
\[
c(i_2,\cdots,i_k,n) =
\frac{ (n-i_2-2i_3- \cdots -(k-1)i_k )!}{
i_2! \cdots i_k! (n-2i_2-3i_3- \cdots -(ki_k )!}.
\]
\end{theorem}
For the case $k=3$ we get the following corollary.
\begin{corollary}\label{c1}
(i) Let $A\in M_{3}(K)$
and let $X^3 = tX^2 - sX + d$ denote the characteristic polynomial of $A$.
Then, for all $n \geq 3$,
\begin{equation}\label{eq1}
A^n = a_{n-1}A + a_{n-2} Adj(A) + (a_n - ta_{n-1})\,I,
\end{equation}
where
\[
a_n =
\sum_{2i+3j \leq n} (-1)^i \binom{i+j}{ j} \binom{n-i-2j }{i+j} t^{n-2i-3j}s^id^j
\]
for $n>0$ and $a_0=1$.
\end{corollary}
We  use this corollary in conjunction with Khovanskii's ideas to
determine sequences of rational approximations to  the real root
of certain types of polynomials.

\section{Approximating Cuberoots of Positive Integers}
We next prove Theorem \ref{t3}.
\begin{theorem}\label{t3}
Let $\alpha>1$  be an
integer and $a$ be any integer such that
\begin{equation}\label{acond}
a>-\frac{{\alpha }^{2/3}}
    {1 + {\alpha }^{1/3}}.
\end{equation}
Set
\begin{multline*}
a_n=\\
\sum_{
i,j
}(-1)^i\binom{i+j}{ j}\binom{n-i-2j }{i+j}(3a)^{n-2i-3j}(3a^{2}-3\alpha)^i
(a^3+\alpha -3a\alpha + \alpha^{2})^j.
\end{multline*}
Then
\begin{equation}\label{ratiolim}
\lim_{n \to \infty} 1+
\frac{\alpha -1}
{\displaystyle{
 \frac{a_{n}}{a_{n-1}}-a+1 } }=
\alpha^{1/3}.
\end{equation}
\end{theorem}

\begin{proof}
Let $\omega := \exp ( 2 \pi \imath/3)$ and set
\[
A=
\left (
\begin{matrix}
a & \alpha &\alpha\\
1&a&\alpha\\
1&1&a
\end{matrix}
\right).
\]
The eigenvalues of $A$ are
\begin{align}\label{betas}
\beta_{1}&= a + {\alpha }^{1/3} + {\alpha }^{2/3},\\
 \beta_{2}&= a + {\alpha }^{1/3}\,\omega  + {\alpha }^{2/3}\,{\omega }^2, \notag\\
\beta_{3}&=  a + {\alpha }^{2/3}\,\omega  + {\alpha }^{1/3}\,{\omega }^2. \notag
\end{align}
Note that $\beta_{1}$ is positive for any $a$ satisfying \eqref{acond}. Further, for such $a$,
\begin{equation*}
\left |
\frac{\beta_{2}}{\beta_{1}}
\right |^{2}
=
 \left |
\frac{\beta_{3}}{\beta_{1}}
\right |^{2}
=
\frac{\beta_{2}\beta_{3}}{\beta_{1}^{2}}
=
\frac{a^2 - a\,{\alpha }^{1/3} +
    {\alpha }^{2/3} -
    a\,{\alpha }^{2/3} - \alpha  +
    {\alpha }^{4/3}}{{\left( a +
       {\alpha }^{1/3} +
       {\alpha }^{2/3} \right) }^2}<1,
\end{equation*}
so that $\beta_{1}>|\beta_{2}|=|\beta_{3}|$.
Let
{\allowdisplaybreaks
\begin{align*}
&M=
\left (
\begin{matrix}
 {\alpha }^{2/3} & {\alpha }^{2/3}\,{\omega }^2 & {\alpha }^
    {2/3}\,\omega  \\
{\alpha }^{1/3} & {\alpha }^{1/3}\,
   \omega  & {\alpha }^{1/3}\,{\omega }^2 \\
1 & 1 & 1
\end{matrix}
\right ), &&
D=
\left (
\begin{matrix}
 \beta_{1}& 0 & 0 \\
0& \beta_{2} & 0 \\
0 & 0 & \beta_{3}
\end{matrix}
\right ).&
\end{align*}
}
Then $A=M\,D\, M^{-1}$ and so
\begin{multline*}
A^{n}=M\,D^{n}\, M^{-1}
=\\
\left (
\begin{matrix}
 \frac{{{{\beta }_1}}^n + {{{\beta }_2}}^n + {{{\beta }_3}}^n}{3}&
   \frac{{\alpha }^{1/3}\,\left( {{{\beta }_1}}^n + \omega \,{{{\beta }_2}}^n
+\omega^{2}  \,{{{\beta }_3}}^n \right) }{3}&
   \frac{{\alpha }^{2/3}\,\left( {{{\beta }_1}}^n
        +\omega^{2}  \,{{{\beta }_2}}^n + \omega \,{{{\beta }_3}}^n \right) }{3} \\
&&\\
     \frac{{{{\beta }_1}}^n +\omega^{2}  \,{{{\beta }_2}}^n +
      \omega \,{{{\beta }_3}}^n}{3\,{\alpha }^{1/3}}&
   \frac{{{{\beta }_1}}^n + {{{\beta }_2}}^n + {{{\beta }_3}}^n}{3}&
   \frac{{\alpha }^{1/3}\,\left( {{{\beta }_1}}^n + \omega \,{{{\beta }_2}}^n
        +\omega^{2}  \,{{{\beta }_3}}^n \right) }{3} \\
&&\\
   \frac{{{{\beta }_1}}^n + \omega \,{{{\beta }_2}}^n
      +\omega^{2}  \,{{{\beta }_3}}^n}{3\,{\alpha }^{2/3}}&
   \frac{{{{\beta }_1}}^n +\omega^{2}  \,{{{\beta }_2}}^n +
      \omega \,{{{\beta }_3}}^n}{3\,{\alpha }^{1/3}}&
   \frac{{{{\beta }_1}}^n + {{{\beta }_2}}^n + {{{\beta }_3}}^n}{3}
\end{matrix}
\right).
\end{multline*}
Let $A_{n,\,i,\,j}$ denote the $(i,j)$-th entry of $A^{n}$. It is now easy to see
(since $\beta_{1}>|\beta_{2}|=|\beta_{3}|$) that
\begin{equation}\label{Anlim}
\lim_{n \to \infty}
\frac{A_{n,\,2,\,1}}
{A_{n,\,3,\,1}}
=
\alpha^{1/3} \lim_{n \to \infty}
\frac{{{{\beta }_1}}^n + \omega^{2}  \,{{{\beta }_2}}^n +\omega  \, {{{\beta }_3}}^n}
{{{{\beta }_1}}^n +\omega \,{{{\beta }_2}}^n +
      \omega^{2}  \,{{{\beta }_3}}^n}
=\alpha^{1/3}.
\end{equation}

On the other hand, the characteristic polynomial of $A$ is
\[
X^{3}=3\,a\,X^{2}-(3 a^{2}-3 \alpha)X+ a^3+\alpha -3 a \alpha + \alpha^{2}.
\]
It follows from Corollary \ref{c1}, that if
$t=3a$, $s=3 a^{2}-3 \alpha$, $d = a^3+\alpha -3 a \alpha + \alpha^{2}$
and
{\allowdisplaybreaks
\begin{align*}
a_n &=
\sum_{2i+3j \leq n} (-1)^i \binom{i+j}{ j} \binom{n-i-2j }{i+j} t^{n-2i-3j}s^id^j,\\
\gamma_{n}&:=\left( a^3 + \alpha  - 3\,a\,\alpha  + {\alpha }^2 \right) \,{a_{n-3 }} -
  2\,\left( a^2 - \alpha  \right) \,{a_{n-2 }} + a\,{a_{n-1 }},\\
\delta_{n}&:=\left( -a + \alpha  \right) \,{a_{n-2 }} + {a_{n-1 }},\\
\rho_{n}&:=\left( 1 - a \right) \,{a_{n-2 }} + {a_{n-1 }},\\
\end{align*}
}
then
\begin{equation*}
A^{n}=
\left (
\begin{matrix}
 \gamma_{n}& \alpha \rho_{n} & \alpha  \delta_{n} \\
\delta_{n}&  \gamma_{n}& \alpha \rho_{n}  \\
 \rho_{n}& \delta_{n} &\gamma_{n}
\end{matrix}
\right ).
\end{equation*}
Thus \eqref{ratiolim} now follows
 by comparing $\lim_{n \to \infty}\delta_{n}/\rho_{n}$ with the limit found above.
\end{proof}
Remarks:
(a) Note that the limit in \eqref{ratiolim} is independent of the choice of $a$,
so that various corollaries can be obtained from
particular choices of $a$.\\
(b) A similar method can be used to approximate square roots and
roots of
higher order (see   Section 4).\\
(c) The pairs $(1,2)$ and $(1,3)$ in \eqref{Anlim} can be replaced by
other pairs to give limits of the form $\alpha^{j/3}$, $-4 \leq j \leq 4$.

\begin{corollary}
Let $\alpha$ be a positive integer. Set
\begin{equation*}
a_n=
\sum_{
2i+3j\leq n
}\binom{i+j}{ j}\binom{n-i-2j }{i+j}3^{n-i-3j}(\alpha -1)^{i+2j}.
\end{equation*}
Then
\begin{equation*}
\lim_{n \to \infty} 1+
(\alpha -1)
\displaystyle{
 \frac{a_{n-1}}{a_{n}} }=
\alpha^{1/3}.
\end{equation*}
\end{corollary}
\begin{proof}
Let $a=1$ in Theorem \ref{t3}.
\end{proof}

\begin{corollary}
Let $\alpha$ be a positive integer. Set
\begin{align*}
a_n=
\sum_{
i=0
}^{n}\binom{2n+i }{2n-2i}3^{3i}\alpha^{2n+i}(\alpha +1)^{2n-2i},\\
b_n=
\sum_{
i=0
}^{n-1}\binom{2n+i }{2n-2i-1}3^{3i+1}\alpha^{2n+i}(\alpha +1)^{2n-2i-1}.
\end{align*}
Then
\begin{equation*}
\lim_{n \to \infty} 1+
\frac{\alpha -1}{
\displaystyle{
 \frac{a_{n}}{b_{n}} }+1}=
\alpha^{1/3}.
\end{equation*}
\end{corollary}
\begin{proof}
Let $a=0$ and replace $n$ by $6n$ in Theorem \ref{t3}.
\end{proof}

\begin{corollary}
Let $\alpha$ be a positive integer. Set
\begin{equation*}
a_n=
\sum_{
i=0
}^{\lfloor n/3 \rfloor}
\binom{n-2i }{i}3^{n-3i}\alpha^{n-i}(\alpha -1)^{2i}.
\end{equation*}
Then
\begin{equation*}
\lim_{n \to \infty} 1+
\frac{\alpha^{2} -1}{
\displaystyle{
 \frac{a_{n}}{a_{n-1}} }-\alpha +1}=
\alpha^{2/3}.
\end{equation*}
\end{corollary}
\begin{proof}
Replace $\alpha$ by $\alpha^{2}$ and then let $a=\alpha$ in Theorem \ref{t3}.
\end{proof}

It is clear from \eqref{Anlim} that the smaller the ratios
$|\beta_{2}/\beta_{1}| =|\beta_{3}/\beta_{1}| $, the
faster will be the rate of convergence in \eqref{ratiolim}. It
is also clear from \eqref{betas} that these ratios can be made
arbitrarily close to 1 by choosing $a$ arbitrarily large. We are interested
in how small this ratio can be made (to get fastest convergence) and what is the
optimal choice of $a$ for a given $\alpha$ to produce this smallest ratio.

\begin{theorem}\label{tdiff}
Let $\alpha$ and $a$ be as described in Theorem \ref{t3}. Let
the matrix $A$ be as described at \eqref{mateq}. Then  the choice of $a$ which
gives the most rapid convergence is one of the two integers closest to
\begin{equation}\label{aeq}
\bar{a}=  \frac{{\alpha }^{1/3} + \alpha }
  {1 + {\alpha }^{1/3}}.
\end{equation}
For this choice of $a$ and $n \geq 3$,
\begin{equation}\label{Anapproxim}
\frac{A_{n,\,2,\,1}}
{A_{n,\,3,\,1}}-\,\alpha^{1/3}
=\left(
\left(   \omega  -1\right) \,\omega \,
    \left( {\left(\frac{ -\omega}{2}  \right) }^n - {\left( \frac{-{\omega }^2 }{2}\right) }^n \right)
+\frac{n\delta_{3}}{2^{n}\alpha^{1/3}}
+\frac{\delta_{4}}
{2^{2n}}
\right)\alpha^{1/3},
\end{equation}
where $\omega = \exp (2 \pi \imath/3)$, $|\delta_{3}|\leq 8$ and $|\delta_{4}|\leq 48$.
\end{theorem}
\begin{proof}
For the moment we consider $a$ to be a real variable and define
\begin{equation*}
h(a) =
\left |
\frac{\beta_{2}}{\beta_{1}}
\right |^{2}
=
 \left |
\frac{\beta_{3}}{\beta_{1}}
\right |^{2}
=
\frac{\beta_{2}\beta_{3}}{\beta_{1}^{2}}
=
\frac{a^2 - a\,{\alpha }^{1/3} +
    {\alpha }^{2/3} -
    a\,{\alpha }^{2/3} - \alpha  +
    {\alpha }^{4/3}}{{\left( a +
       {\alpha }^{1/3} +
       {\alpha }^{2/3} \right) }^2}.
\end{equation*}
The function $h(a)$ achieves its minimum at
\[
a=\bar{a}:=\frac{{\alpha }^{1/3} + \alpha }
  {1 + {\alpha }^{1/3}} \text{ and }
h(\bar{a})=
\frac{{\left( -1 + {\alpha }^{1/3} \right) }^2}
  {4\,\left( 1 + {\alpha }^{1/3} +
      {\alpha }^{2/3} \right) }.
\]
Hence for large $\alpha$ the best possible choice of $a$ is one of
the two integers closest to $\bar{a}$, say
\[
a'=\frac{{\alpha }^{1/3} + \alpha }
  {1 + {\alpha }^{1/3}} + \eta,
\]
with $|\eta| < 1$.
With this choice,
\begin{align*}
\frac{\beta_{2}\beta_{3}}{\beta_{1}^{2}}
&=\frac{1}{4}
-\frac{3}{4}\frac{\left( 1 +
         {\alpha }^{1/3} +
         {\alpha }^{2/3} \right)  4\,\alpha  +
      \left( 1 + {\alpha }^{1/3} \right) \,
       \eta \,\left( 4\,{\alpha }^{2/3} -
         \eta  - {\alpha }^{1/3}\,\eta
         \right)   }
{{\left( 2\,
        \left( 1 + {\alpha }^{1/3} +
          {\alpha }^{2/3} \right) \,
        {\alpha }^{1/3} +
       \left( 1 + {\alpha }^{1/3} \right) \,
        \eta  \right) }^2}\\
&=:\frac{1}{4}
+\frac{3}{4}g(\eta).
\end{align*}
Next, considering $g(\eta)$ as a function of $\eta$,
\[
g'(\eta)=\frac{4\,{\left( 1 + {\alpha }^{1/3} \right)
        }^4\,{\alpha }^{1/3}\,\eta }{{\left(
       2\,{\alpha }^{1/3} +
       2\,{\alpha }^{2/3} + 2\,\alpha  +
       \eta  + {\alpha }^{1/3}\,\eta  \right)
      }^3}
\]
Thus, since $g'(0)=0$ and $g(1)<g(-1)$,
\begin{align*}
&\frac{1}{4}
+\frac{3}{4}g(0)\leq
\frac{\beta_{2}\beta_{3}}{\beta_{1}^{2}}
\leq \frac{1}{4}
+\frac{3}{4}g(-1),
\end{align*}
or
{\allowdisplaybreaks
\begin{multline*}
\frac{1}{4}-
\frac{3}{4}
\frac{\,{\alpha }^{1/3}}
  {\,\left( 1 + {\alpha }^{1/3} +
      {\alpha }^{2/3} \right) }
\leq
\frac{\beta_{2}\beta_{3}}{\beta_{1}^{2}}\leq
\\
 \frac{1}{4}-
\frac{3}{4}
\frac{\,\left( -1 + {\alpha }^{1/3} \right)
      \,\left( 1 + 3\,{\alpha }^{1/3} +
      8\,{\alpha }^{2/3} + 8\,\alpha  +
      4\,{\alpha }^{4/3} \right) }
{\,
    {\left( -1 + {\alpha }^{1/3} +
        2\,{\alpha }^{2/3} + 2\,\alpha
        \right) }^2}.
\end{multline*}
}

Thus $\beta_{2}\beta_{3}/\beta_{1}^{2}<1/4$ or $|\beta_{2}/\beta_{1}|=
|\beta_{3}/\beta_{1}|<1/2$, for $\alpha >1$.
\begin{align}\label{deleq}
\frac{\beta_{2}}{\beta_{1}}&
=
 \frac{{\alpha }^{1/3} -
       {\alpha }^{2/3} + \eta  +
       {\alpha }^{1/3}\,\eta  +
       {\alpha }^{1/3}\,\omega  -
       \alpha \,\omega }
{2\,{\alpha }^{1/3} +
       2\,{\alpha }^{2/3} + 2\,\alpha  +
       \eta  + {\alpha }^{1/3}\,\eta }
\\
&
 = -\frac{\omega}{2}+
\frac{\delta_{1}}{\alpha^{1/3}}
, \notag  \\
\frac{\beta_{3}}{\beta_{1}}&
=\frac{-{\alpha }^{2/3} + \alpha  + \eta  +
     {\alpha }^{1/3}\,\eta  -
     {\alpha }^{1/3}\,\omega  +
     \alpha \,\omega }{2\,{\alpha }^{1/3} +
     2\,{\alpha }^{2/3} + 2\,\alpha  + \eta  +
     {\alpha }^{1/3}\,\eta } \notag
\\
& =
-\frac{\omega^{2}}{2}+ \frac{\delta_{2}}{\alpha^{1/3}}, \notag
\end{align}
where $|\delta_{1}|$, $|\delta_{2}| < 1$. (We omit the details of these calculations.
The first equation is simply solved for $\delta_{1}$,  the solution is multiplied by its
conjugate $\overline{\delta_{1}^{}}=\delta_{2}$,
the resulting real number is shown to be monotone
decreasing as a function of $\eta$ by differentiating with respect to $\eta$,
and finally it is shown that $\delta_{1}\overline{\delta_{1}^{}}<1$ at $\eta=-1$.)

Note that these
ratios $|\beta_{2}/\beta_{1}|=
|\beta_{3}/\beta_{1}|$ increase quite slowly with $\alpha$: $|\beta_{2}/\beta_{1}|<0.45$,
for $\alpha < 3000$, for example.
Returning to large $\alpha$,
{\allowdisplaybreaks
\begin{align}\label{Anapprox}
\frac{A_{n,\,2,\,1}}
{A_{n,\,3,\,1}}&-\,\alpha^{1/3}
=
 \left (
\frac{{{{\beta }_1}}^n + \omega^{2}  \,{{{\beta }_2}}^n +\omega  \, {{{\beta }_3}}^n}
{{{{\beta }_1}}^n +\omega \,{{{\beta }_2}}^n +
      \omega^{2}  \,{{{\beta }_3}}^n}-1
\right )\,\alpha^{1/3}\\
&=\frac{\left( 1 - \omega  \right) \,\omega \,
      \left( -{{{\beta }_2}}^n + {{{\beta }_3}}^n \right) }{{{{\beta }_1}}^n +
      \omega \,{{{\beta }_2}}^n + {\omega }^2\,{{{\beta }_3}}^n} \alpha^{1/3}\notag \\
&=\left(
\left(   \omega  -1\right) \,\omega \,
    \left( {\left(\frac{ -\omega}{2}  \right) }^n - {\left( \frac{-{\omega }^2 }{2}\right) }^n \right)
+\frac{n\delta_{3}}{2^{n}\alpha^{1/3}}
+\frac{\delta_{4}}
{2^{2n}}
\right)\alpha^{1/3},\notag
\end{align}
}
where $|\delta_{3}|\leq 8$ and $|\delta_{4}|\leq 48$. Note that
we have used \eqref{deleq} to replace the ratios $\beta_{2}/\beta_{1}$ and
 $\beta_{3}/\beta_{1}$ in the final expression.
\end{proof}

Remark: Note that for $n\geq 3$ and $\alpha>2^{4n}$,  we have
the following:
\begin{equation}
2^{n}
\left (\frac{A_{n,\,2,\,1}}
{A_{n,\,3,\,1}\,\alpha^{1/3}}-1
\right)
=
\begin{cases}
-3+\frac{K_{n}}{2^{n}}, & n \equiv 1,\,2\, (\mod 6),\\
0+\frac{K_{n}}{2^{n}}, & n \equiv 3,\,6 \,(\mod 6),\\
3+\frac{K_{n}}{2^{n}}, & n \equiv 4,\,5 \,(\mod 6),\\
\end{cases}
\end{equation}
where $|K_{n}|<61$.

\section{Approximating the Real Root of an Arbitrary Cubic}
If the zeros of $a\,x^3+b x^2+c x+ d$ are $\beta_{1}$, $\beta_{2}$ and
$\beta_{3}$, then the zeros of
$x^{3}+(9ac-3 b^{2})x+ 2b^{3} - 9abc + 27a^{2}d$ are
$3a\beta_{1}+b$, $3a\beta_{2}+b$ and
$3a\beta_{3}+b$. Thus, in finding the roots of a general cubic equation,
it is sufficient to study cubics of the form $f(x)=x^3-px-q$.
For simplicity, here we restrict to the case $p>0$, $q>0$ and $27q^{2}-4p^{3}>0$,
so that $f(x)$ has exactly one real root, which is largest in absolute value.
We have the following theorem.
\begin{theorem}\label{tpq}
Let $p>0$, $q>0$ be integers such that $27q^{2}-4p^{3}>0$. Define
\[
a_{n}=\sum_{
2i+3j\leq n
}\binom{i+j}{ j}\binom{n-i-2j }{i+j}3^{n-2i-3j}(3-p)^{i}(q-p+1)^{j}.
\]
Then
\begin{equation}
-1+\lim_{n \to \infty}\frac{a_{n}}{a_{n-1}}=
\frac{{\left( 2/3\right) }^{1/3}\,p}
   {{\left( 9\,q +{\sqrt{81\,q^2-12\,p^3 }} \right) }^{1/3}} +
  \frac{{\left( 9\,q + {\sqrt{81\,q^2-12\,p^3 }} \right) }^{1/3}}
   {2^{1/3}\,3^{2/3}},
\end{equation}
the real root of $f(x)=x^{3}-px-q$.
\end{theorem}

\begin{proof}
As before, let $\omega = \exp ( 2 \pi \imath/3)$ and set
\[
A=
\left (
\begin{matrix}
1 & p & q \\
1&1& 0 \\
0&1&1
\end{matrix}
\right).
\]
Define
\begin{align*}
&\alpha=\frac{{\left( 2/3\right) }^{1/3}\,p}
   {{\left( 9\,q -{\sqrt{81\,q^2-12\,p^3 }} \right) }^{1/3}},& &\beta =
  \frac{{\left( 9\,q - {\sqrt{81\,q^2-12\,p^3 }} \right) }^{1/3}}
   {2^{1/3}\,3^{2/3}}.&
\end{align*}
The eigenvalues of $A$ are
\begin{align}\label{gammas}
\gamma_{1}&= 1 + {\alpha } +\beta,\\
 \gamma_{2}&= 1 + {\alpha }\,\omega^{2}  + {\beta }\,{\omega }, \notag\\
\gamma_{3}&=  1 + {\alpha }\,\omega  + {\beta }\,{\omega }^2. \notag
\end{align}
Set \[
M=\left (
\begin{matrix}
 (  \alpha  + \beta )^{2}  & ( \beta \,\omega  + \alpha \,{\omega }^2 )^{2}
& (  \alpha \,\omega  + \beta \,{\omega }^2 )^{2} \\
 \alpha  + \beta  & \beta \,\omega  + \alpha \,{\omega }^2 & \alpha \,\omega  + \beta \,{\omega }^2 \\
 1  & 1 & 1
\end{matrix}
\right)
\]
and then
\[
M^{-1}A\,M=\left (
\begin{matrix}
\gamma_{1} & 0 & 0 \\
0&\gamma_{2}& 0 \\
0&0&\gamma_{3}
\end{matrix}
\right).
\]
Here we use the facts that $q=\alpha^{3}+\beta^{3}$ and $p = 3 \alpha \beta$.
Clearly
\[
A^{n}=M\left (
\begin{matrix}
\gamma_{1}^{n} & 0 & 0 \\
0&\gamma_{2}^{n} & 0 \\
0&0&\gamma_{3}^{n}
\end{matrix}
\right)M^{-1}.
\]

As before, let $A_{n,\,i\,j}$ denote the $(i,\,j)$ entry of $A^{n}$.
It is straightforward to show (preferably after using a computer algebra system like
\emph{Mathematica} to perform the
matrix multiplications) that
\begin{align*}
A_{n,2,1}&=\frac{\left( -1 + {{\gamma }_1} \right) \,{{{\gamma }_1}}^n}
   {\left( {{\gamma }_1} - {{\gamma }_2} \right) \,
     \left( {{\gamma }_1} - {{\gamma }_3} \right) } +
  \frac{\left( -1 + {{\gamma }_3} \right) \,{{{\gamma }_3}}^n}
   {\left( {{\gamma }_1} - {{\gamma }_3} \right) \,
     \left( {{\gamma }_2} - {{\gamma }_3} \right) } +
  \frac{\left( -1 + {{\gamma }_2} \right) \,{{{\gamma }_2}}^n}
   {\left( {{\gamma }_1} - {{\gamma }_2} \right) \,
     \left( -{{\gamma }_2} + {{\gamma }_3} \right) } ,\\
A_{n,3,1}&=\frac{-\left( {{{\gamma }_2}}^n\,{{\gamma }_3} \right)  +
    {{\gamma }_2}\,{{{\gamma }_3}}^n +
    {{{\gamma }_1}}^n\,\left( -{{\gamma }_2} + {{\gamma }_3} \right)  +
    {{\gamma }_1}\,\left( {{{\gamma }_2}}^n - {{{\gamma }_3}}^n \right) }
    {\left( {{\gamma }_1} - {{\gamma }_2} \right) \,
    \left( {{\gamma }_1} - {{\gamma }_3} \right) \,
    \left( -{{\gamma }_2} + {{\gamma }_3} \right) }.
\end{align*}
Since $\gamma_{1}>|\gamma_{2}|$, $|\gamma_{3}|$, it follows that
\begin{equation}\label{pqlim}
\lim_{n \to \infty} \frac{A_{n,\,2,\,1}}{A_{n,\,3,\,1}}=\gamma_{1}-1 = \alpha + \beta.
\end{equation}
Next, the real zero of $x^{3}-p\,x-q=0$ is
\[
\frac{{\left( 2/3\right) }^{1/3}\,p}
   {{\left( 9\,q +{\sqrt{81\,q^2-12\,p^3 }} \right) }^{1/3}} +
  \frac{{\left( 9\,q + {\sqrt{81\,q^2-12\,p^3 }} \right) }^{1/3}}
   {2^{1/3}\,3^{2/3}},
\]
and  some simple algebraic manipulation shows that
this is equal to $\alpha + \beta$, so that
the limit at \eqref{pqlim} is indeed equal to this real zero.

Finally, the characteristic polynomial of $A$ is
\[
X^{3}=3\,X^{2}-(3-P)X+q+1-p,
\]
so that Corollary \ref{c1} gives, after setting $t=3$, $d=q+1-p$ and $s=3-p$,
\[
a_{n}=\sum_{
2i+3j\leq n
}\binom{i+j}{ j}\binom{n-i-2j }{i+j}3^{n-2i-3j}(3-p)^{i}(q-p+1)^{j},
\]
and \[
\epsilon_{n}=\left( 1 - p + q \right) \,{a_{n-3 }} +
    \left( -2 + p \right) \,{a_{n-2 }} + {a_{n-1 }},
\]
that
\[
A^{n}=\left (
\begin{matrix}
\epsilon_{n}& \left(  q - p  \right) \,{a_{n-2}} + p\,{a_{n-1}} &q\,\left( {a_{n-1}} -{a_{n-2}}  \right)  \\
{a_{n-1}} -{a_{n-2}} & \epsilon_{n} &  q\,{a_{n-2}}\\
 {a_{n-2}}&  {a_{n-1}}-{a_{n-2}} &
\epsilon_{n}-p\,a_{n-2}
\end{matrix}
\right).
\]
The result now follows, after comparing $\lim_{\to \infty}A_{n,\,2,\,1}/A_{n,\,3,\,1}$ in the
matrix above with the limit found at \eqref{pqlim}.
\end{proof}

\section{Approximating roots of Arbitrary order of a positive integer}
Khovanskii shows that the method of section 2
 extends to roots of arbitrary  order $m$, by considering the $m \times m$
matrix
\begin{equation}\label{matneq}
A=
\left (
\begin{matrix}
a         & \alpha           &\alpha &\alpha &\dots &\alpha\\
1         &  a                  &\alpha&\alpha &\dots &\alpha\\
1         &          1         &a       &\alpha &\dots &\alpha\\
\vdots & \vdots           & \vdots & \ddots & \vdots & \vdots\\
1   & 1                  &  1        & \dots   &  a      & \alpha \\
1   & 1                  &  1        & \dots   &  1      & a
\end{matrix}
\right).
\end{equation}
Again his result is dependent on the existence of $\lim_{n \to \infty}
A_{n,\,i,\,j}/A_{n,\,u,\,v}$, for various pairs $(i,j)$ and $(u,v)$, but
he does not suggest any criteria which guarantee these limits exist.
We make his statement more precise in the following theorem.

\begin{theorem}\label{tm}
Let $A$ be the matrix defined above at \eqref{matneq}. Let $A_{n,\,i,\,j}$
denote the $(i,\,j)$ entry of $A^{n}$ and suppose $a>0$. Then
\begin{equation}\label{matnlim}
\lim_{n \to \infty}
\frac{A_{n,\,i,\,j}}
{A_{n,\,u,\,v}}
= \alpha^{(j+u-i-v)/m}.
\end{equation}
\end{theorem}

\begin{proof}
Let $\omega_{m}$ be a primitive $m$-th root of unity. Define the matrix $M$ by
\[
(M)_{i,\,j}= \alpha^{(m-i)/m}\omega_{m}^{(m-j+1)i}.
\]
Then
\[
(M^{-1})_{i,j}=\frac{1}{m\,(M)_{j,\,i}}.
\]
(We omit the proof of this statement. It can easily be checked
by  showing that multiplying
$M$ and the claimed inverse together gives the $m\times m$ identity matrix.)

It is now not difficult  to show that
\[
M^{-1}A\,M= \text { diag } (\beta_{1},\beta_{2}, \dots \beta_{m}),
\]
where \text { diag } $(\beta_{1},\beta_{2}, \dots \beta_{m}) $ is the matrix with
$\beta_{1},\beta_{2}, \dots \beta_{m}$ along the main diagonal and zeroes elsewhere. Here
\[
\beta_{i}=a+\sum_{j=1}^{m-1}(\omega_{m}^{i-1}\alpha^{1/m})^{j}, \hspace{25pt} i=1,2,\dots m,
\]
 are the eigenvalues of
$A$.
For $a>0$ , there is clearly a dominant eigenvalue, namely
$\beta_{1}$.

(This condition could be relaxed to allow $a$ to take some negative values,
but the precise lower bound which makes $\beta_{1}>|\beta_{j}|$, $j\not =1$,
is not so easy to determine in the case of arbitrary $m$.)

Next, it is clear that
$A^{n} = M \text{diag }(\beta_{1}^{n},\beta_{2}^{n}, \dots \beta_{m}^{n})M^{-1}$,
and it is simple algebra to show that
{\allowdisplaybreaks
\[
A_{n,\, i \, j} =\frac{\alpha^{(j-i)/m}}{m}
\sum_{k=1}^{m}\omega_{m}^{(m-k+1)(i-j)}\beta_{k}^{n}.
\]
}

The result now follows, upon using the fact that $\beta_{1}$ is the dominant
eigenvalue.
\end{proof}

Note, as in Theorem \ref{t3}, that the limit is independent of the choice of $a$.

Theorem \ref{t2} could be use to produce results
similar to those in Theorem \ref{t3} and
its various corollaries, but the statements of these results become much
more complicated
with increasing $m$.

Also, we have not been able to determine the optimum choice of
$a$ that gives the most rapid convergence in \eqref{matnlim}. One
difference between the $m=3$ case and the general case is that
the sub-dominant eigenvalues in the general case need not necessarily
all have the same absolute value.

\section{Concluding Remarks}
For completeness we include the following  neat construction
by Khovanskii, one that enables
good approximations to a root of an arbitrary
polynomial to be found in many cases.
Let
\begin{multline*}
A=\\
\left (
\begin{matrix}
   k      &   l\,a_{m}         & 0   &   \dots      &   0      &   0         & 0   &  0  \\
  0      &   k        &  l\,a_{m}    &   \dots      &   0      &   0         & 0   &  0  \\
   \vdots     &        \vdots   & \vdots       &  \vdots &    \vdots &  \vdots  &  \vdots  &  \vdots  \\
   0      &   0         & 0   &     \dots     &    l\,a_{m}       &   0         & 0   &  0  \\
  0      &   0         & 0   &  \dots         &   k      &    l\,a_{m}         & 0   &  0  \\
  0      &   0         & 0   &   \dots        &   0      &   k         & 0   &   l\,a_{m}   \\
  -l\,a_{0}  & -l\,a_{1} & -l\,a_{2}& \dots  &-l\,a_{m-4} & -l\,a_{m-3} & k-l\,a_{m-1}& -l\,a_{m-2}\\
  0      &   0         & 0   &    \dots        &   0      &   0         & l\,a_{m}  &  k
\end{matrix}
\right).
\end{multline*}

Here $k$ and $l$ are non-zero.
 If $\lim_{n \to \infty}A_{n,\,i,\,1}/A_{n,\,m,\,1}$ exists
and equals, say, $\beta_{i}$, for
$1 \leq i \leq m$, then $\beta_{m-1}$ is a root of
\[
f(x)=a_{n}x^{n}+a_{n-1}x^{n-1}+a_{n-2}x^{n-2}\dots a_{1}x+a_{0}.
\]

This can be seen as  follows. Since the limits exist and $\beta_{m}=1$,
we get the system of equations
{\allowdisplaybreaks
\begin{align*}
\beta_{i}&=\frac{k\,\beta_{i}+l\,a_{m}\beta_{i+1}}{l\,a_{m}\beta_{m-1}+k},
\hspace{40pt}1 \leq i \leq m-3,\\
\beta_{m-2}&=\frac{k\,\beta_{m-2}+l\,a_{m}}{l\,a_{m}\beta_{m-1}+k}\\
\beta_{m-1}&=\frac{-l\,a_{0}\beta_{1}-l\,a_{1}\beta_{2}- \dots-l\,a_{m-3}\beta_{m-2}+
(k-l\,a_{m-1})\beta_{m-1}-l\,a_{m-2}}
{l\,a_{m}\beta_{m-1}+k}.
\end{align*}
}
This system of equations leads to the system $\beta_{m-1}\beta_{m-2}=1$,
$\beta_{i+1}=\beta_{m-1}\beta_{i}$, $1\leq i \leq m-3$, and
\[
a_{m}\beta_{m-1}^{2}+a_{m-1}\beta_{m-1}+a_{m-2}+a_{m-3}\beta_{m-2}
+ \dots +a_{1}\beta_{2} +a_{0}\beta_{1}=0.
\]
The result now follows, after multiplying the last equation by $\beta_{m-1}^{m-2}$ and
using the equations preceding it to eliminate $\beta_{i}$, $i \not = m-1$.

This situation is of course even more difficult to analyze: $f(x)$ may not even
have real zeroes, or it may have multiple real zeroes, or even if it has
 a single real zero, this may not  be enough
to guarantee that the limits $\lim_{n \to \infty}A_{n,\,i,\,1}/A_{n,\,m,\,1}$,
$1 \leq i \leq m$,  exist,.

It would be interesting to find and prove general criteria, based
on the entries of the matrix $A$, which guarantee that this method
of Khovanskii's does lead to  convergence to one of the roots.



\end{document}